\documentclass[12pt]{amsart}
\usepackage{amsmath}
\usepackage{amsthm}
\usepackage{amssymb}
\usepackage{latexsym}
\usepackage{amscd}
\usepackage[all]{xy}
\usepackage{mathrsfs}

\newtheorem{theorem}{Theorem}  
\newtheorem{lemma}[theorem]{Lemma}
\newtheorem{proposition}[theorem]{Proposition}

\newtheorem{corollary}[theorem]{Corollary}

\newtheorem{remar}[theorem]{Remark}
\renewenvironment{proof}{Proof:\ \ \ }{\QED}
\newenvironment{remark}{\begin{remar}\rm}{\end{remar}}

\theoremstyle{definition}

\newcommand{\QED}{{\unskip\nobreak\hfil\penalty50%
\hskip1em\hbox{}\nobreak\hfil $\Box$%
\parfillskip=0pt \finalhyphendemerits=0 \par\medskip\noindent}}
\newcommand{\bfind}[1]{\index{#1}{\bf #1}}

\newcommand{\n}{\par\noindent}

\newcommand{\sn}{\par\smallskip\noindent}

\newcommand{\bn}{\par\bigskip\noindent}
\newcommand{\pars}{\par\smallskip}
\newcommand{\parm}{\par\medskip}
\newcommand{\parb}{\par\bigskip}
\newcommand{\card}{\mbox{\rm card}}
\newcommand{\On}{\mbox{\rm On}}
\newcommand{\cf}{\mbox{\rm cf}}
\newcommand{\ci}{\mbox{\rm ci}}
\newcommand{\Coin}{\mbox{\rm Coin}}
\newcommand{\Cofin}{\mbox{\rm Cofin}}
\newcommand{\supp}{\mbox{\rm supp}\,}
\newcommand{\Rl}{R_{\mbox{\tiny\rm left}}}
\newcommand{\Rr}{R_{\mbox{\tiny\rm right}}}
\newcommand{\phil}{\varphi_{\mbox{\tiny\rm left}}}
\newcommand{\phir}{\varphi_{\mbox{\tiny\rm right}}}
\newcommand{\isom}{\simeq}
\newcommand{\zero}{\underline{0}}

\newcommand{\R}{\mathbb R}

\newcommand{\N}{\mathbb N}
\newcommand{\Z}{\mathbb Z}
\newcommand{\Reg}{{\rm Reg}}
\newcommand{\cal}{\mathcal}
\newcommand{\pH}{\mathop{\raisebox{-.2ex}{\rule[-1pt]{0pt}{1pt}%
\mbox{\large\bf H}}}}
\newcommand{\ovl}[1]{\overline{#1}}

\begin{document}
\title[Symmetrical completeness]{Symmetrically complete ordered
sets, abelian groups and fields}
\author{Katarzyna \& Franz-Viktor Kuhlmann, Saharon Shelah}
\address{Department of Mathematics \& Statistics, University of
Saskatchewan, 106 Wiggins Road, Saskatoon, SK, S7N 5E6, Canada}
\email{fvk@math.usask.ca}
\address{Institute of Mathematics, Silesian University, Bankowa 14,
40-007 Katowice, Poland}
\email{kmk@math.us.edu.pl}
\address{Department of Mathematics,
The Hebrew University of Jerusalem,
Jeru\-salem, Israel}
\email{shelah@math.huji.ac.il}

\date{August 1, 2013}
\thanks{The research of the second author was partially supported
by a Canadian NSERC grant and a sabbatical grant from the University
of Saskatchewan.\\
The third author would like to thank the Israel Science Foundation for
partial support of this research (Grant no. 1053/11). Paper 1024 on his
publication list.}

\begin{abstract}
We characterize and construct linearly ordered sets, abelian groups and
fields that are {\emph symmetrically complete}, meaning that the
intersection over any chain of closed bounded intervals is nonempty.
Such ordered abelian groups and fields are important because
generalizations of Banach's Fixed Point Theorem hold in them.
We prove that symmetrically complete ordered abelian groups and
fields are divisible Hahn products and real closed power series fields,
respectively. We show how to extend any given ordered set, abelian group
or field to one that is symmetrically complete. A main part of the paper
establishes a detailed study of the cofinalities in cuts.
\end{abstract}
\maketitle


%
%
\section{Introduction}
In the paper \cite{[S]}, the third author introduced the notion of
``symmetrically complete'' ordered fields and proved that every ordered
field can be extended to a symmetrically complete ordered field. He also
proved that an ordered field $K$ is symmetrically complete if and only
if every nonempty chain of closed bounded intervals in $K$ has a
nonempty intersection; talking of chains of intervals, we refer to the
partial ordering by inclusion. It is this property that is particularly
interesting as it allows to prove fixed point theorems for such fields
that generalize Banach's Fixed Point Theorem, replacing the usual metric
of the reals by the distance function that is derived from the ordering,
see \cite{[KK]}. In accordance with the notion used in that paper, we
will call a linearly ordered set $(I,<)$ \bfind{spherically complete
w.r.t.\ the order balls} if every nonempty chain of closed bounded
intervals has a nonempty intersection. Ordered fields and ordered
abelian groups shall be called spherically complete w.r.t.\ the order
balls if the underlying linearly ordered set is.

\pars
It is not a priori clear whether fields that are spherically complete
w.r.t.\ the order balls, other than the reals themselves, do exist. At
first glance, the above condition on chains of intervals seems to imply
that the field is cut complete and hence isomorphic to the reals. But
it is shown in \cite{[S]} that there are arbitrarily large fields with
this property. Let us describe the background in more detail; for some
of the notions used, see Section~\ref{sectprel}.

A \bfind{cut} in a linearly ordered set $I$ is a pair $C=(D,E)$ with a
\bfind{lower cut set} $D$ and an \bfind{upper cut set} $E$ if $I=D\cup
E$ and $d<e$ for all $d\in D$, $e\in E$. Throughout this paper, when we
talk of cuts we will mean \bfind{Dedekind cuts}, that is, cuts with $D$
and $E$ nonempty. By the \bfind{cofinality} of the cut $C$ we mean the
pair $(\kappa,\lambda)$ where $\kappa$ is the cofinality of $D$, denoted
by $\cf (D)$, and $\lambda$ is the \bfind{coinitiality} of $E$, denoted
by $\ci (E)$. Recall that the coinitiality of a linearly ordered set is
the cofinality of this set under the reversed ordering. Recall further
that cofinalities and coinitialities of ordered sets are regular
cardinals.

We will call a linearly ordered set $(I,<)$ \bfind{symmetrically
complete} if every cut $C$ in $I$ is \bfind{asymmetric}, that is,
$\kappa\ne\lambda$. Ordered fields and ordered abelian groups shall be
called symmetrically complete if the underlying linearly ordered set is.
For example, the reals are symmetrically complete because every cut $C$
is \bfind{principal} (also called \bfind{realized}), that is, either $D$
has a maximal element or $E$ has a minimal element, in which case the
cofinality of $C$ is either $(1,\aleph_0)$ or $(\aleph_0,1)$. We have
the following characterization of symmetrical completeness, which will
be proved in Section~\ref{sectprel}:

\begin{proposition}                         \label{scobcas}
A linearly ordered set $I$ is spherically complete w.r.t.\ the order
balls if and only if every nonprincipal cut in $I$ is asymmetric.
\end{proposition}

Note that $\Z$ has these properties, but a discretely ordered abelian
group is never symmetrically complete. On the other hand, there are no
cuts with cofinality $(1,1)$ in densely ordered abelian groups and in
ordered fields, so they are symmetrically complete as soon as every
nonprincipal cut is asymmetric.

\pars
The main aim of this paper is to characterize the symmetrically complete
ordered abelian groups and fields. Obviously, an ordered field is
symmetrically complete if and only if its underlying additive
ordered abelian group is. But ordered abelian groups also appear as the
value groups of nonarchimedean ordered fields w.r.t.\ their natural
valuations.

More generally, we will need the \bfind{natural valuation} of any
ordered abelian group $(G,<)$, which we define as follows. Two elements
$a,b\in G$ are called \bfind{archimedean equivalent} if there is some
$n\in\N$ such that $n|a|\geq |b|$ and $n|b|\geq |a|$. The ordered
abelian group $(G,<)$ is archimedean ordered if all nonzero elements are
archimedean equivalent. If $0\leq a<b$ and $na<b$ for all $n\in\N$, then
``$a$ is infinitesimally smaller than $b$'' and we will write $a\ll b$.
We denote by $va$ the archimedean equivalence class of $a$. The set of
archimedean equivalence classes can be ordered by setting $va>vb$ if and
only if $|a|<|b|$ and $a$ and $b$ are not archimedean equivalent, that
is, if $n|a|<|b|$ for all $n\in\N$. We write $\infty:=v0\,$; this is the
maximal element in the linearly ordered set of equivalence classes. The
function $a\mapsto va$ is a group valuation on $G$, i.e., it satisfies
$va=\infty\Leftrightarrow a=0$ and the ultrametric triangle law
\sn
{\bf (UT)} \ $v(a-b)\>\geq\min\{va,vb\}\,$,
\sn
and by definition,
\[
0\leq a\leq b\>\Longrightarrow\>va\geq vb\;.
\]
The set $vG:=\{vg\mid 0\ne g\in G\}$ is called the \bfind{value set} of
the valued abelian group $(G,v)$. For every $\gamma\in vG$, the quotient
${\cal C}_\gamma:= {\cal O}_\gamma/{\cal M}_\gamma$, where
${\cal O}_\gamma:=\{g\in G\mid vg\geq\gamma\}$ and ${\cal M}_\gamma:=
\{g\in G\mid vg>\gamma\}$, is an archimedean ordered abelian group
(hence embeddable in the ordered additive group of the reals, by the
Theorem of H\"older); it is called an \bfind{archimedean component} of
$G$. The natural valuation induces an ultrametric given by $u(a,b):=
v(a-b)$.

We define the smallest \bfind{ultrametric ball} $B_u(g,h)$ containing
the elements $g$ and $h$ to be
\[
B_u(a,b)\>:=\>\{g\mid v(a-g)\geq v(a-b)\}\>=\>
\{g\mid v(b-g)\geq v(a-b)\}
\]
where the last equation holds because in an ultrametric ball, every
element is a center. For the basic facts on ultrametric spaces, see
\cite{[KU3]}. Note that all ultrametric balls are cosets of convex
subgroups in $G$ (see \cite{[KU4]}).
%
%
We say that an ordered abelian group (or an ordered field) is
\bfind{spherically complete w.r.t.\ its natural valuation} if every
nonempty chain of ultrametric balls (ordered by inclusion) has a
nonempty intersection. The ordered abelian groups that are spherically
complete w.r.t.\ their natural valuation are precisely the Hahn products
(see \cite{[KU4]} or \cite{[KS]}); see Section~\ref{secthahn} for the
definition and basic properties of Hahn products.

\pars
If $(K,<)$ is an ordered field, then we consider the natural
valuation on its ordered additive group and define $va+vb:=v(ab)$.
This turns the set of archimedean classes into an ordered abelian group,
with neutral element $0:=v1$ and inverses $-va=v(a^{-1})\,$. In this
way, $v$ becomes a field valuation (with additively written value
group). It is the finest valuation on the field $K$ which is compatible
with the ordering. The residue field, denoted by $Kv$, is archimedean
ordered, hence by the version of the Theorem of H\"older for ordered
fields, it can be embedded in the ordered field $\R$. Via this
embedding, we will always identify it with a subfield of $\R$.

\begin{remark}
In contrast to the notation for the natural valuation (in the Baer
tradition) that we have used in \cite{[KK]}, we use here the Krull
notation because it is more compatible with our constructions in
Section~\ref{sectcon}. In this notation, two elements in an ordered
abelian group or field are close to each other when the value of their
difference is large.
\end{remark}

Every ordered field that is spherically complete
w.r.t.\ its natural valuation is maximal, in the sense of \cite{[Ka]}.
In this paper Kaplansky shows that under certain conditions, which in
particular hold when the residue field has characteristic $0$, every
such field is isomorphic to a power series field. In general, a
nontrivial factor system is needed on the power series field, but it is
not needed for instance when the residue field is $\R$.

In \cite{[KK]}, we have already proved that if an ordered abelian group
$(G,<)$ is spherically complete w.r.t.\ the order balls, then it is
spherically complete w.r.t.\ its natural valuation $v$. If $G$ is even
an ordered field, then we proved that in addition, it has residue field
$\R$. From this and Proposition~\ref{scobcas}, we obtain:

\begin{proposition}                               \label{scoscu}
If an ordered abelian group is symmetrically complete,
then it is spherically complete w.r.t.\ its natural valuation.
If an ordered field is symmetrically complete,
then it is spherically complete w.r.t.\ its natural valuation $v$ and
has residue field $Kv=\R$.
\end{proposition}

In the present paper, we wish to extend these results. It turns out that
for an ordered abelian group to be symmetrically complete, the same must
be true for the value set $vG$, and in fact, it must have an even
stronger property. We will call a cut with cofinality $(\kappa,\lambda)$
in a linearly ordered set $(I,<)$ \bfind{strongly asymmetric} if
$\kappa\ne\lambda$ and at least one of $\kappa,\,\lambda$ is
uncountable. We will call $(I,<)$ \bfind{strongly symmetrically
complete} if every cut in $I$ is strongly asymmetric, and we will call
it \bfind{extremely symmetrically complete} if in addition, the
coinitiality and cofinality of $I$ are both uncountable. The reals are
not strongly symmetrically complete.

\pars
In Section~\ref{sectprf}, we will prove the following results:

\begin{theorem}                             \label{MTag}
An ordered abelian group $(G,<)$ is symmetrically complete if and
only if it is spherically complete w.r.t.\ its natural valuation $v$, has
a strongly symmetrically complete value set $vG$ and all archimedean
components ${\cal C}_\gamma$ are isomorphic to $\R$. It is strongly
symmetrically complete if and only if in addition, $vG$ has uncountable
cofinality, and it is extremely symmetrically complete if and only if in
addition, $vG$ is extremely symmetrically complete.
\end{theorem}

Now we turn to ordered fields.

\begin{theorem}                             \label{MTf}
An ordered field $K$ is symmetrically complete if and only if it is
spherically complete w.r.t.\ its natural valuation $v$, has residue
field $\R$ and a strongly symmetrically complete value group $vK$.
Further, the following are equivalent:
\sn
a) $K$ is strongly symmetrically complete,
\sn
b) $K$ is extremely symmetrically complete,
\sn
c) $K$ is
spherically complete w.r.t.\ its natural valuation $v$, has residue
field $\R$ and an extremely symmetrically complete value group $vK$.
\end{theorem}

\begin{corollary}                           \label{HP,PSF}
Every symmetrically complete ordered abelian group is divisible and
isomorphic to a Hahn product. Every symmetrically complete ordered
field is real closed and isomorphic to a power series field with
residue field $\R$ and divisible value group.
\end{corollary}

These results show a way for the construction of symmetrically complete
and extremely symmetrically complete ordered fields $K$, which is an
alternative to the construction given in \cite{[S]}. For the former,
construct a strongly symmetrically complete linearly ordered set $I$
with uncountable coinitiality. Then take $G$ to be the Hahn product with
index set $I$ and all archimedean components equal to $\R$. Finally,
take $K=\R((G))$, the power series field with coefficients in $\R$ and
exponents in $G$. To obtain an extremely symmetrically complete ordered
field $K$, construct $I$ such that in addition, also its cofinality is
uncountable. See Section~\ref{sectcon} for details.

\pars
In Section~\ref{sectemb}, we will use our theorems to prove the
following result, which extends the corresponding result of \cite{[S]}:

\begin{theorem}                             \label{ext}
Every ordered abelian group can be extended to an extremely
symmetrically complete ordered abelian group. Every ordered field can be
extended to an extremely symmetrically complete ordered field.
\end{theorem}

For the proof of this theorem, we need to extend any given ordered set
$I$ to an extremely symmetrically complete ordered set $J$. We do this
by constructing suitable lexicographic products of ordered sets. Let us
describe the most refined result that we achieve, which gives us the
best control of the cofinalities of cuts in the constructed ordered set
$J$.

We denote by $\Reg$ the class of all infinite regular cardinals, and for
any ordinal $\lambda$, by
\[
\Reg_{<\lambda} \>=\> \{\kappa<\lambda\mid \aleph_0\leq\kappa =
\cf(\kappa)\}
\]
the set of all infinite regular cardinals $<\lambda$. We define:
\begin{eqnarray*}
\Coin (I)&:=&\{\ci(S)\mid S\subseteq I \mbox{ such that }\ci(S)
\mbox{ is infinite}\}\>\subset\>\Reg\,,\\
\Cofin (I)&:=&\{\cf(S)\mid S\subseteq I \mbox{ such that }\cf(S)
\mbox{ is infinite}\}\>\subset\>\Reg\,.
\end{eqnarray*}

We choose any $\mu,\kappa_0,\lambda_0\in\Reg$. Then we set
\begin{eqnarray*}
\Rl & := & \Cofin (I)\cup\Reg_{<\kappa_0}\cup\Reg_{<\mu}\>\subset\>
\Reg\>,\\
\Rr & := & \Coin (I)\>\cup \Reg_{<\lambda_0}\cup \Reg_{<\mu}\>\subset\>
\Reg\>.
\end{eqnarray*}
All of the subsets we have defined here are initial segments of $\Reg$
in the sense that if they contain $\kappa$, then they also contain every
infinite regular cardinal $<\kappa$.

\pars
Further, we assume that functions
\[
\phil: \{1\}\cup\Reg\rightarrow\Reg\;\mbox{ and }\;\phir:
\{1\}\cup\Reg\rightarrow\Reg
\]
are given.
We prove in Section~\ref{sectcon}:

\begin{theorem}                             \label{IextJ}
Assume that $\mu$ is uncountable and that
\begin{equation}                            \label{phi}
\phil(\{1\}\cup\Rr)\subset\Rl\;\mbox{ and }\;\phir(\{1\}\cup\Rl)\subset\Rr
\end{equation}
with $\phil(\kappa)\ne\kappa\ne\phir(\kappa)$ for all $\kappa\in\Rl\cup
\Rr$. Then $I$ can be extended to a strongly symmetrically complete
ordered set $J$ of cofinality $\kappa_0$ and coinitiality $\lambda_0\,$,
in which the cuts have the following cofinalities:
\[
\{(1,\mu),(\mu,1)\} \,\cup\, \{(\kappa,\varphi(\kappa))\mid \kappa\in
\Rl\} \,\cup\, \{(\varphi(\lambda),\lambda)\mid\lambda\in \Rr\}\>.
\]
If in addition $\kappa_0$ and $\lambda_0$ are uncountable, then $J$ is
extremely symmetrically complete.
\end{theorem}

\parm
Among the value groups of valued fields, not only the dense, but also
the discretely ordered groups play an important role. The value groups
of formally $p$-adic fields are discretely ordered, and the value groups
of $p$-adically closed fields are $\Z$-groups, that is, ordered abelian
groups $G$ that admit (an isomorphic image of) $\Z$ as a convex subgroup
such that $G/\Z$ is divisible. We wish to prove a version of the
previous theorem for discretely ordered abelian groups. Note that in a
discretely ordered group, every principal cut has cofinality $(1,1)$. We
call an ordered abelian group $G$ \bfind{symmetrically d-complete} if
every nonprincipal cut in $G$ is asymmetric. We will call it
\bfind{extremely symmetrically d-complete} if in addition, $G$ has
uncountable cofinality (and hence also uncountable coinitiality).
Note that if a nonprincipal cut is asymmetric, then it is
strongly asymmetric because the only countable coinitiality/cofinality
other than 1 is $\aleph_0\,$. So ``symmetrically d-complete'' is at the
same time the discrete version of ``strongly symmetrically complete''.

\begin{theorem}                             \label{MTagd}
For a discretely ordered abelian group $(G,<)$, the following are
equivalent:
\sn
a) \ $(G,<)$ is symmetrically d-complete,
\sn
b) \ $(G,<)$ is spherically complete w.r.t.\ the order balls,
\sn
c) \ $(G,<)$ is a $\Z$-group such that $G/\Z$ is strongly
symmetrically complete.
\sn
%
%
%
Further, $(G,<)$ is extremely symmetrically d-complete if and only if
$G/\Z$ is extremely symmetrically complete.
%
\end{theorem}

Again, this shows a way of construction. To obtain a symmetrically
d-complete discretely ordered abelian group $G$, construct a
strongly symmetrically complete ordered abelian group $H$ and then take
the lexicographic product $H\times\Z$. If in addition the cofinality of
$H$ is uncountable, then $G$ will even be extremely symmetrically
d-complete.

\pars
Note that $G$ is isomorphic to a Hahn product if and only if $G/\Z$ is.
Therefore, if $G$ is symmetrically d-complete, then it is a Hahn
product.

\parm
\begin{remark}
After the completion of this paper it was brought to our attention by
Salma Kuhlmann that Hausdorff constructed already in the years 1906-8
ordered sets with prescribed cofinalities for all of its cuts (cf.\
\cite{[Hd]}). His construction also yields extremely symmetrically
complete ordered sets. However, we are convinced that the constructions
we present in Section~\ref{sectcon} are indispensable in the context of
this paper (and a good service to the reader), for the following
reasons:
\sn
$\bullet$ \ they represent a shortcut to the result we need, whereas the
construction of Hausdorff is complicated and spread over several
sections of the long paper \cite{[Hd]}; moreover, it is written in
German and in a somewhat oldfashioned notation that is not always the
most elegant (according to our ``modern'' standards);
\sn
$\bullet$ \ Hausdorff does not construct the ordered sets so that they
extend a given ordered set; convincing the reader that his
construction can be adapted to accommodate this additional condition
would essentially take the same effort as a driect proof.
\end{remark}

%
%
\section{Preliminaries and notations}                  \label{sectprel}
%
%
%
\subsection{Proof of Proposition~\ref{scobcas}}
A \bfind{quasicut} in a linearly ordered set $I$ is a pair $C=(D,E)$
of subsets $D$ and $E$ of $I$ such that $I=D\cup E$ and $d\leq e$ for
all $d\in D$, $e\in E$. In this case, $D\cap E$ is empty or a singleton;
if it is empty, then $(D,E)$ is a cut.

\pars
Assume that every nonprincipal cut in the linearly ordered set $I$ is
asymmetric. Every nonempty chain of closed bounded intervals has a
cofinal subchain $([d_\nu,e_\nu])_{\nu<\mu}$ indexed by a regular
cardinal $\mu$. We set $D:=\{d\in I\mid d\leq d_\nu\mbox{ for some }
\nu<\mu\}$ and $E:=\{e\in I\mid e\geq e_\nu\mbox{ for some }\nu<\mu\}$.
Then $d\leq e$ for all $d\in D$ and $e\in E$. If $D\cap E\ne\emptyset$,
then $(D,E)$ is a quasicut and the unique element of $D\cap E$ lies in
the intersection of the chain. If $(D,E)$ is a cut, then because of
$\cf(D)=\mu=\ci(E)$ it must be principal, i.e., $\mu=1$ and
$\{d_0,e_0\}=[d_0,e_0]$ is contained in the intersection of the chain.
If $D\cap E=\emptyset$ but $(D,E)$ is not a cut, then the set $\{c\in
I\mid d<c<e\mbox{ for all }d\in D,\,e\in E\}$ is nonempty and contained
in the intersection of the chain. So in all cases, the intersection of
the chain is nonempty.

Now assume that $I$ is spherically complete w.r.t.\ the order balls.
Suppose that $(D,E)$ is a cut with $\kappa:=\cf(D)=\ci(E)$. Then we
can choose a cofinal sequence $(d_\nu)_{\nu<\kappa}$ in $D$ and a
coinitial sequence $(e_\nu)_{\nu<\kappa}$ in $E$. By assumption, the
descending chain $([d_\nu\,,\,e_\nu])_{\nu<\kappa}$ of intervals has
nonempty intersection. But this is only possible if $\kappa=1$, which
implies that $(D,E)$ is principal. This proves that every
nonprincipal cut is asymmetric.

%
%
\subsection{Hahn products}                  \label{secthahn}
Given a linearly ordered index set $I$ and for every $\gamma\in I$ an
arbitrary abelian group $C_\gamma\,$, we define a group called the
\bfind{Hahn product}, denoted by $\pH_{\gamma\in I} C_\gamma\,$.
Consider the product $\prod_{\gamma\in I} C_\gamma$ and an element $c=
(c_\gamma)_{\gamma \in I}$ of this group. Then the \bfind{support} of
$c$ is the set $\supp c:=\{\gamma\in I\mid c_\gamma \not=0\}$. As a
set, the Hahn product is the subset of $\prod_{\gamma\in I} C_\gamma$
containing all elements whose support is a wellordered subset of $I$,
that is, every nonempty subset of the support has a minimal element). In
particular, the support of every nonzero element $c$ in the Hahn product
has a minimal element $\gamma_0\,$, which enables us to define a group
valuation by setting $vc=\gamma_0$ and $v0=\infty$. The Hahn product is
a subgroup of the product group. Indeed, the support of the sum of two
elements is contained in the union of their supports, and the union of
two wellordered sets is again wellordered.

We leave it to the reader to show that a Hahn product is divisible if
and only if all of its components are.

If the components $C_\gamma$ are (not necessarily archimedean) ordered
abelian groups, we obtain the \bfind{ordered Hahn product}, also called
\bfind{lexicographic product}, where the ordering is defined as
follows. Given a nonzero element $c= (c_\gamma)_{\gamma \in I}\,$, let
$\gamma_0$ be the minimal element of its support. Then we take $c>0$ if
and only if $c_{\gamma_0}>0$. If all $C_\gamma$ are archimedean
ordered, then the valuation $v$ of the Hahn product coincides with the
natural valuation of the ordered Hahn product. Every ordered abelian
group $G$ can be embedded in the Hahn product with its set of
archimedean classes as index sets and its archimedean components as
components. Then $G$ is spherically complete w.r.t.\ the ultrametric
balls if and only if the embedding is onto.

%
%
\subsection{Some facts about cofinalities and
coinitialities}                                   \label{sectcoco}
Take a nontrivial ordered abelian group $G$ and define
\[
G^{>0}\>:=\>\{g\in G\mid g>0\}\;\;\mbox{ and }\;\;
G^{<0}\>:=\>\{g\in G\mid g<0\}\>.
\]
Since $G\ni g\mapsto -g\in G$ is an order inverting bijection,
\[
\ci(G)\>=\>\cf(G)\;\;\mbox{ and }\;\;\cf(G^{<0})\>=\>\ci(G^{>0})\>.
\]

Further, we have:
\begin{lemma}                               \label{gco}
1) \ The cofinality of\/ $G$ is equal to $\max\{\aleph_0,\ci(vG)\}$.
Hence it is uncountable if and only if the coinitiality of $vG$ is
uncountable.
\sn
2) \ If\/ $G$ is discretely ordered, then $\ci(G^{>0})=\cf(vG)=1$.
Otherwise, $\ci(G^{>0})=\max\{\aleph_0,\cf(vG)\}$.
\sn
3) \ Take $\gamma\in vG$, not the largest element of $vG$, and let
$\kappa$ be the coinitiality of the set $\{\delta\in vG\mid\delta>
\gamma\}$. Then $\cf({\cal M}_\gamma)=\max\{\aleph_0,\kappa\}$.
\end{lemma}
\begin{proof}
1): \
 Since a nontrivial
ordered abelian group has no maximal element, its cofinality is at least
$\aleph_0\,$. If $vG$ has a smallest element, then take a positive $g\in
G$ whose value is this smallest element. Then the sequence
$(ng)_{n\in\N}$ is cofinal in $G$, so its cofinality is $\aleph_0\,$.

If $\kappa:=\ci(vG)$ is infinite, then take a sequence
$(\gamma_\nu)_{\nu<\kappa}$ which is coinitial in $vG$, and take
positive elements $g_\nu\in G$, $\nu<\kappa$, with
$vg_\nu=\gamma_\nu\,$. Then the sequence $(g_\nu)_{\nu<\kappa}$ is
cofinal in $G$ and therefore, $\cf(G)\leq\ci(vG)$. On the other hand,
for every sequence $(g_\nu)_{\nu<\lambda}$ cofinal in $G$, the sequence
of values $(vg_\nu)_{\nu<\lambda}$ must be coinitial in $vG$, which
shows that $\cf(G)\geq\ci(vG)$.

\sn
2) \ If\/ $G$ is discretely ordered, then it has a smallest positive
element $g$ and hence, $\ci(G^{>0})=1$. Further, $vg$ must be the
largest element of $vG$, so $\cf(vG)=1$.

If $G$ is not discretely, hence densely ordered, then the coinitiality
of $G^{>0}$ is at least $\aleph_0\,$. If $vG$ has a largest element
$\gamma$, then we take a positive $g\in G$ with $vg=\gamma$. Then
${\cal M}_\gamma=\{0\}$ and ${\cal O}_\gamma$ is an archimedean
ordered convex subgroup of $G$. This implies that $\ci(G^{>0})=
\ci({\cal O}_\gamma^{>0})= \aleph_0\,$.

If $\kappa:=\cf(vG)$ is infinite, then take a sequence
$(\gamma_\nu)_{\nu<\kappa}$ which is cofinal in $vG$, and take positive
elements $g_\nu\in G$, $\nu<\kappa$, with $vg_\nu=\gamma_\nu\,$. Then
the sequence $(g_\nu)_{\nu<\kappa}$ is coinitial in $G^{>0}$ and
therefore, $\ci(G^{>0})\leq\cf(vG)$. On the other hand, for every
sequence $(g_\nu)_{\nu<\lambda}$ coinitial in $G^{>0}$, the sequence of
values $(vg_\nu)_{\nu<\lambda}$ must be coinitial in $vG$, which shows
that $\ci(G^{>0})\geq\cf(vG)$.

\sn
3): \ By our condition on $\gamma$, ${\cal M}_\gamma$ is a nontrivial
subgroup of $G$ and therefore, its cofinality is at least $\aleph_0\,$.
If $v{\cal M}_\gamma=\{\delta\in vG\mid\delta> \gamma\}$ has a smallest
element, then take a positive $g\in G$ whose value is this smallest
element. Then the sequence $(ng)_{n\in\N}$ is cofinal in
${\cal M}_\gamma$, so $\cf({\cal M}_\gamma)=\aleph_0\,$.

Assume that $\kappa=\ci(v{\cal M}_\gamma)$ is infinite. Take a sequence
$(\gamma_\nu)_{\nu<\kappa}$ which is coinitial in $v{\cal M}_\gamma=
\{\delta\in vG\mid\delta> \gamma\}$ and take positive elements $g_\nu
\in G$, $\nu<\kappa$, with $vg_\nu=\gamma_\nu\,$. Then the sequence
$(g_\nu)_{\nu<\kappa}$ is cofinal in ${\cal M}_\gamma$ and therefore,
$\cf({\cal M}_\gamma)\leq\kappa$. On the other hand, for every sequence
$(g_\nu)_{\nu<\lambda}$ cofinal in ${\cal M}_\gamma$, the sequence of
values $(vg_\nu)_{\nu<\lambda}$ must be coinitial in $v{\cal M}_\gamma$,
which shows that $\cf({\cal M}_\gamma)\geq\kappa$.
\end{proof}

%

%
%
\section{Analysis of cuts in ordered abelian groups}
In this section, we will use the facts outlined in
Section~\ref{sectcoco} freely without further citation.

Take a cut $C=(D,E)$ with cofinality $(\kappa,\lambda)$ in the ordered
abelian group $G$. First assume that $C$ is principal. If $D$ has
largest element $g$, then the set $g+G^{>0}$ is coinitial in $E$. Hence
in this case, $C$ has cofinality $(1,\lambda)$ with $\lambda=\ci
(G^{>0})$. Symmetrically, if $E$ has smallest element $g$, then the set
$g+G^{<0}$ is cofinal in $D$. Hence in this case, $C$ has cofinality
$(\kappa,1)$ with $\kappa=\cf(G^{<0})=\ci(G^{>0})$.

If $G$ is discretely ordered, then $\ci(G^{>0})=1$ by part 2) of
Lemma~\ref{gco}. So for every principal cut to be asymmetric, it is
necessary that $G$ is not discretely, hence densely ordered. If $G$ is
densely ordered, then $\ci(G^{>0})=\max\{\aleph_0,\cf(vG)\}$. So we
obtain:

\begin{lemma}                               \label{prin}
Take any ordered abelian group $G$. Every principal cut in $G$ is
asymmetric if and only it $G$ is densely ordered. Every principal cut in
$G$ is strongly asymmetric if and only if in addition, $\cf(vG)$ is
uncountable.
\end{lemma}

\pars
From now on we assume that the cut $C$ in $G$ is nonprincipal. Then the
only countable cardinality that can appear as coinitiality or cofinality
is $\aleph_0\,$. This shows:

\begin{lemma}                               \label{asa}
If a nonprincipal cut is asymmetric, then it is strongly asymmetric.
\end{lemma}

We consider the ultrametric balls $B_u(d,e)$ for all $d\in D,\,e\in E$.
Any two of them have nonempty intersection since this intersection will
contain both a final segment of $D$ and an initial segment of $E$. Since
two ultrametric balls with nonempty intersection are already comparable
by inclusion, it follows that these balls form a nonempty chain.
Now there are two cases:
\sn
I) the chain contains a smallest ball,
\n
II) the chain does not contain a smallest ball.
\sn
First, we discuss cuts of type I). We choose $d_0\in D,\,e_0\in E$ such
that $B_u(d_0,e_0)$ is the smallest ball. The shifted cut
\[
C-d_0\>:=\>(\{d-d_0\mid d\in D\}\,,\,\{e-d_0\mid e\in E\})
\]
has the same cofinality as $C$. Moreover,
\[
B_u(d_0,e_0)-d_0\>:=\> \{b-d_0\mid b\in B_u(d_0,e_0)\}\>=\>
B_u(0,e_0-d_0)
\]
remains the smallest ball in the new situation. Therefore, we can assume
that $d_0=0$. Set $\gamma:=ve_0$ and $I:= [0,e_0]$. Then $vh\geq \gamma$
for all $h\in I$, that is, $h\in {\cal O}_\gamma$. The images $D'$ of
$D\cap I$ and $E'$ of $E\cap I$ in ${\cal C}_\gamma ={\cal O}_\gamma/
{\cal M}_\gamma$ are convex and satisfy $D'\leq E'$. If there were
$d'\in D'\cap E'$, then it would be the image of elements $d\in D\cap
I$, and $e\in E\cap I$, with $\gamma< v(e-d)$, and $B_u(d,e)$ would be a
ball properly contained in $B_u(0,e_0)$, contrary to our minimality
assumption. Hence, $D'<E'$. If there were an element strictly between
$D'$ and $E'$, then it would be the image of an element $h-d_0$ with $h$
strictly between $D$ and $E$, which is impossible. So we see that
$(D',E')$ defines a cut $C'$ in ${\cal C}_\gamma\,$, with $D'$ a
final segment of the left cut set and $E'$ an initial segment of the
right cut set.

Since ${\cal C}_\gamma$ is archimedean ordered, the cofinality of $C'$
can only be $(1,1)$, $(1,\aleph_0)$, $(\aleph_0,1)$, or $(\aleph_0,
\aleph_0)$. Lifting cofinal sequences in $D'$ back into $D$, we see that
if the cofinality of $D'$ is $\aleph_0\,$, then so is the cofinality of
$D$. Similarly, if the coinitiality of $E'$ is $\aleph_0\,$, then so is
the coinitiality of $E$. However, if $D'$ contains a last element $a'$,
and if $a\in D\cap I$ is such that $a$ has image $a'$ in
${\cal C}_\gamma\,$, then the set of all elements in $G$ that are sent
to $a'$ is exactly the coset $a+ {\cal M}_\gamma$. This set has empty
intersection with $E$ since $a'\notin E'$. This together with $a'$
being the last element of $D'$ shows that $a+ {\cal M}_\gamma$ is a
final segment of $D$ and therefore, the cofinality of $D$ is equal to
that of ${\cal M}_\gamma$. Similarly, if $E'$ has a first element $b'$
coming from an element $b\in E\cap I$, then $b+{\cal M}_\gamma$ is
an initial segment of $E$ and therefore, the coinitiality of $E$ is
equal to that of ${\cal M}_\gamma$, which in turn is equal to the
cofinality of ${\cal M}_\gamma$. If $\lambda$ denotes this cofinality, we
see that the cofinality of $C$ is
\sn
a) \ $(\lambda,\lambda)$ if $C'$ has cofinality $(1,1)$,\n
b) \ $(\lambda,\aleph_0)$ or $(\aleph_0,\lambda)$ if $C'$ has cofinality
$(1,\aleph_0)$ or $(\aleph_0,1)$, and \n
c) \ $(\aleph_0, \aleph_0)$ if $C'$ has cofinality $(\aleph_0,\aleph_0)$.

\pars
Cofinality $(1,1)$ can only appear for $C'$ if ${\cal C}_\gamma$ is
isomorphic to $\Z$, and then every cut in ${\cal C}_\gamma$
has this cofinality. In this case, $C$ is principal (and thus ``out of
scope'' in our present discussion) if and only if $\lambda=1$, which
means that ${\cal M}_\gamma=\{0\}$ and thus, $\gamma$ is the maximal
element of $vG$ and ${\cal O}_\gamma\isom {\cal C}_\gamma\isom\Z$,
showing that $G$ is discretely ordered.

Cofinality $(\aleph_0,\aleph_0)$ can only appear (and will appear) for
$C'$ if ${\cal C}_\gamma$ has nonprincipal cuts. If $(D'_1,E'_1)$ is a
cut in ${\cal C}_\gamma$ with cofinality $(\aleph_0,\aleph_0)$, then we
set
\begin{eqnarray*}
D_1 & := & \{d\in G\mid d\leq d_1\mbox{ for some }d_1\in{\cal O}_\gamma
\mbox{ with }d_1+{\cal M}_\gamma\in D'_1\}\>,\\
E_1 & := & \{e\in G\mid e\geq e_1\mbox{ for some }e_1\in{\cal O}_\gamma
\mbox{ with }e_1+{\cal M}_\gamma\in E'_1\}\>.
\end{eqnarray*}
This defines a cut in $G$ with cofinality $(\aleph_0,\aleph_0)$.

We conclude that the only choice for the components ${\cal C}_\gamma$
that prevents cofinality $(\lambda,\lambda)$, $\lambda\geq\aleph_0\,$,
for $C$ is: ${\cal C}_\gamma\isom \R$ for all $\gamma\in vG$, or ${\cal
C}_\gamma \isom \Z$ if $\gamma$ is the smallest element of $vG$ and
${\cal C}_\gamma\isom\R$ otherwise.

If this condition is satisfied, then all nonprincipal cuts of type I
have cofinalities $(\lambda,\aleph_0)$ or $(\aleph_0, \lambda)$. Hence all
of them are asymmetric if and only if for all $\gamma\in vG$ not the
last element of $vG$, the cofinality of ${\cal M}_\gamma$ is
uncountable. By part 3) of Lemma~\ref{gco} this happens if and only if
the cut
\[
\gamma^+\>:=\>(\{\delta\in vG\mid\delta\leq\gamma\}\,,\,
\{\delta\in vG \mid\delta>\gamma\})
\]
in $vG$ has cofinality $(1,\lambda)$ with $\lambda$ uncountable.
%
%

Note that every cut in $vG$ of cofinality $(1,\lambda)$ is of the form
$\gamma^+$, in which case the upper cut set will be the value set
$v{\cal M}_\gamma$ of ${\cal M}_\gamma$. Then the cut

\[
(\{d\in G\mid d< {\cal M}_\gamma\},
\{e\in G\mid e\geq c \mbox{ for some }c\in {\cal M}_\gamma\})
\]
in $G$ will have cofinality $(\aleph_0,\lambda')$ with $\lambda'=
\aleph_0$ if $\lambda=1$ and $\lambda'=\lambda$ otherwise. So for every
nonprincipal cut of type I to be asymmetric it is also necessary that
for every cut of cofinality $(1,\lambda)$ in $vG$, $\lambda$ is
uncountable.

\pars
We summarize our discussion so far:

\begin{lemma}                               \label{typeI}
Take any ordered abelian group $G$. Then every nonprincipal cut of type
I is (strongly) asymmetric if and only if the following conditions are
satisfied:
\sn
a) \ ${\cal C}_\gamma\isom\R$ for all $\gamma\in vG$, or
${\cal C}_\gamma \isom \Z$ if $\gamma$ is the largest element of $vG$
and ${\cal C}_\gamma\isom\R$ otherwise.
\sn
b) \ for every cut in $vG$ of cofinality $(1,\lambda)$, $\lambda$ is
uncountable.
\end{lemma}

\parm
Now we discuss nonprincipal cuts $C$ of type II). We assume in addition
that the ordered abelian group $G$ is spherically complete w.r.t.\ its
natural valuation $v$. Then there is some $g\in G$ such that
\[
g\in \bigcap_{d\in D,\,e\in E}^{} B_u(d,e)\;.
\]
Replacing the cut $C$ by the shifted cut $C-g$ as we have done before
already, we can assume that $g=0$. Since $C$ is nonprincipal by
asssumption, there must be $d_0\in D,\,e_0\in E$ such that $d_0\leq
0\leq e_0$ does not hold, and we have two cases:
\sn
A) \ $e_0<0\,$,
\n
B) \ $0<d_0$.
\sn
Again, we set $I:=[d_0,e_0]$. We set $\tilde{D}=\{vd\mid d\in D\cap I\}
\subseteq vG$ and $\tilde{E}=\{ve\mid e\in E\cap I\}\subseteq vG$.
\pars
Let us first discuss case A). We claim that $\tilde{D}<\tilde{E}$. We
observe that ``$\leq$'' holds since $d<e<0$ for $d\in D\cap I$ and
$e\in E\cap I$. Suppose that $\tilde{D}\cap \tilde{E}\ne\emptyset$, that
is, $vd=ve$ for some $d\in D\cap I$ and $e\in E\cap I$. Then $v(e-d)\geq
vd$ by the ultrametric triangle law, and since there is no smallest ball
by assumption, we can even choose $d,e$ such that $v(e-d)>vd$. But then,
$0$ would not lie in $B_u(d,e)$, a contradiction. We have proved
our claim. Now if there were an element $\alpha$ stricly between the two
sets, then there were some $a\in I$ with $va=\alpha$ and $a>0$.
This would yield that $d<a<e$ for all $d\in D\cap I$ and $e\in
E\cap I$ and thus, $D<a<E$, a contradiction.
%

We conclude that $(\tilde{D},\tilde{E})$ defines a cut $\tilde{C}$ in
$vG$, with $\tilde{D}$ a final segment of the left cut set, and
$\tilde{E}$ an initial segment of the right cut set. Denote by
$(\tilde{\kappa},\tilde{\lambda})$ its cofinality. We have that
$vd<ve$ and consequently $vd=v(e-d)$ for all $d\in D\cap I$
and $e\in E\cap I$. Since by assumption there is no smallest ball, there
is no largest value $v(e-d)$. This shows that $\tilde{D}$ has no largest
element and therefore, $\tilde{\kappa}$ is infinite. Lifting cofinal
sequences in $\tilde{D}$ to coinitial sequences in $D$, we see that
$\kappa=\tilde{\kappa}$. By the same argument, if $\tilde{\lambda}$ is
infinite, then $\lambda=\tilde{\lambda}$. If on the other hand
$\tilde{\lambda}=1$, then we take $\gamma\in G$ to be the smallest
element of $\tilde{E}$. The preimage of $\gamma$ under the valuation is
${\cal O}_\gamma\setminus {\cal M}_\gamma$, and this set is coinitial
in $E$. The cofinality of ${\cal O}_\gamma\setminus {\cal M}_\gamma$ is
equal to the cofinality of ${\cal O}_\gamma$, which in turn is equal to
the cofinality $\aleph_0$ of the archimedean ordered group
${\cal C}_\gamma$. Hence in this case, $\lambda=\aleph_0\,$.

From this discussion it follows that $C$ is asymmetric if and only if
$\tilde{C}$ is strongly asymmetric.

\pars
Now we consider case B). Since $0<d<e$ for $d\in D\cap I$ and $e\in E
\cap I$, we now obtain that $\tilde{E} \leq \tilde{D}$. It is proven as
in case A) that $(\tilde{E},\tilde{D})$ defines a cut $\tilde{C}$ in
$vG$, and that $\tilde{E}$ has no largest element. In this case the
argument is the same as before, but with $\tilde{D}$ and $\tilde{E}$
interchanged, and the conclusion is the same as in case A). We note that
in both cases, the cofinality of the left cut set of $C$ must be
infinite.

\pars
If we have a cut $\tilde{C}=(\tilde{D},\tilde{E})$ in $vG$ of cofinality
$(\tilde{\kappa},\tilde{\lambda})$ with $\tilde{\kappa}$ infinite, then
we can associate to it a nonprincipal cut of type II as follows. We set
$D=\{d\in G\mid d<0\mbox{ and } vd\in \tilde{D}\}$ and $E=\{e\in G\mid
e>0\mbox{ or } ve\in \tilde{E}\}$. This is a cut in $G$, and it is of
type II since for all $d\in D$ and $e\in E$, $e<0$, we have that $vd<
ve$ and hence $v(e-d)=vd\in \tilde{D}$, which has no largest element.
Then $C$ induces the cut $\tilde{C}$ in the way described under case A).
Thus for every cut of type II to be asymmetric, it is necessary that
every cut in $vG$ of cofinality $(\tilde{\kappa},\tilde{\lambda})$ with
$\tilde{\kappa}$ infinite is strongly asymmetric.

We summarize:
\begin{lemma}                               \label{typeII}
Take any ordered abelian group $G$ which is spherically complete w.r.t.\
its natural valuation $v$. Then every nonprincipal cut of type II is
asymmetric (and hence strongly asymmetric) if and only if every cut in
$vG$ of cofinality $(\kappa,\lambda)$ with $\kappa$ infinite is
strongly asymmetric.
\end{lemma}

%
%
\section{Proof of the main theorems}        \label{sectprf}
\noindent
{\bf Proof of Theorem~\ref{MTag}}:
\n
Take any ordered abelian group $G$. Assume first that it is
symmetrically complete. Then by Proposition~\ref{scoscu}, $G$ is
spherically complete w.r.t.\ its natural valuation. By Lemma~\ref{prin},
$G$ is densely ordered. Thus it cannot have an archimedean component
${\cal C}_\gamma\isom \Z$ with $\gamma$ the largest element of $vG$
because otherwise, it would have a convex subgroup isomorphic to $\Z$
and would then be discretely ordered. Hence by Lemma~\ref{typeI}, every
archimedean component of $G$ is isomorphic to $\R$ and for every cut in
$vG$ of cofinality $(1,\kappa)$, $\kappa$ is uncountable. Finally by
Lemma~\ref{typeII}, every cut in $vG$ of cofinality $(\lambda,\kappa)$
with $\lambda$ infinite is strongly asymmetric. Altogether, every cut in
$vG$ is strongly asymmetric. This proves that $vG$ is strongly
symmetrically complete.

If $G$ is spherically complete w.r.t.\ its natural valuation, every
archimedean component of $G$ is isomorphic to $\R$ and $vG$ is strongly
symmetrically complete, then in particular, $G$ is densely ordered, and
it follows from Lemmas~\ref{prin}, \ref{typeI} and~\ref{typeII} that $G$
is symmetrically complete.

From Lemma~\ref{asa} we see that for a symmetrically complete $G$ to be
strongly symmetrically complete it suffices that every principal cut is
strongly asymmetric, which by Lemma~\ref{prin} holds if and only if in
addition to the other conditions, the cofinality of $vG$ is uncountable.

%
A strongly symmetrically complete $G$ is extremely symmetrically
complete if and only if in addition, its cofinality (which is equal to
its coinitiality) is uncountable. By part 1) of Lemma~\ref{gco}, this
holds if and only if the coinitiality of $vG$ is uncountable. Hence by
what we have just proved before, a symmetrically complete $G$ is
extremely symmetrically complete if and only if in addition, $vG$ is
extremely symmetrically complete.
\qed

\bn
{\bf Proof of Theorem~\ref{MTf}}:
\n
Considering the additive ordered abelian group of the ordered field $K$,
the first assertion of Theorem~\ref{MTf} follows readily from that of
Theorem~\ref{MTag} if one takes into account that through
multiplication, all archimedean components are isomorphic
to the ordered additive group of the residue field.

Similarly, the equivalence of b) and c) follows from the third case of
Theorem~\ref{MTag}. Since $vK$ is an ordered abelian group, its
cofinality is equal to its coinitiality, so the condition that it is
strongly symmetrically complete with uncountable cofinality already
implies that it is extremely symmetrically complete. Hence, by the
second case of Theorem~\ref{MTag}, a) is equivalent with c).   \qed

\bn
{\bf Proof of Corollary~\ref{HP,PSF}}:
\n
The assertion for ordered abelian groups follows from the facts that
have been mentioned before. For ordered fields, it remains to show that
a power series field with residue field $\R$ and divisible value group
is real closed. Since every power series field is henselian under its
canonical valuation, this follows from \cite[Theorem (8.6)]{[P]}. \qed

\bn
{\bf Proof of Theorem~\ref{MTagd}}:
\n
The equivalence of a) and b) follows directly from
Proposition~\ref{scobcas}. It remains to prove the
equivalence of a) and c).

If $G$ is discretely ordered, then $vG$ must have a largest element
$vg$ (where $g$ can be chosen to be the smallest positive element of
$G$) with archimedean component ${\cal O}_{vg}\isom {\cal C}_{vg}
\isom\Z$. We identify the convex subgroup ${\cal O}_{vg}$ with $\Z$. We
will now prove that $G$ is symmetrically d-complete if and only if
$G/\Z$ is strongly symmetrically complete.

Take any cut $(D,E)$ in $G$. Since the canonical epimorphism
$G\rightarrow G/\Z$ preserves $\leq$, the image $(\ovl{D},\ovl{E})$ of
$(D,E)$ in $G/\Z$ is a quasicut. If $\ovl{D}$ and $\ovl{E}$ have a
common element $\ovl{d}$, then there is $d\in D$ and $z\in\Z$ such that
$d+z\in E$. In this case, the cofinality of $(D,E)$ is $(1,1)$. Now
suppose that $\ovl{D}$ and $\ovl{E}$ have no common element. Then for
all $d\in D$ and $e\in E$, we have that $d+\Z=\{d+z\mid d\in D,\, z\in
\Z\}\subset D$ and $e+\Z\subset E$. Hence if $D'\subset D$ is a set of
representatives for $\ovl{D}$ and $E'\subset E$ is a set of
representatives for $\ovl{E}$, then $D=D'+\Z=\{d+z\mid d\in D',\,
z\in\Z\}$ and $E=E'+\Z$. This yields that
\begin{equation}                            \label{relcof}
\left\{\begin{array}{rclrcl}
\cf (D) &=&\max\{\cf (D'),\aleph_0\}&=&\max\{\cf (\ovl{D}),\aleph_0\}\>,\\
\ci (E) &=&\max\{\ci (E'),\aleph_0\}&=&\max\{\ci (\ovl{E}),\aleph_0\}\>.
\end{array}\right.
\end{equation}

Suppose that $G/\Z$ is strongly symmetrically complete and that $(D,E)$
is a cut in $G$ of cofinality $\ne (1,1)$. Then by what we have just
shown, $(\cf(D),\ci(E))=(\max\{\cf (\ovl{D}),\aleph_0\},\max\{\ci
(\ovl{E}),\aleph_0\})$. By our assumption on $G/\Z$, $(\ovl{D},\ovl{E})$
is strongly asymmetric, which yields that $\cf(D)$ and
$\ci(E)$ are not equal and at least one of them is uncountable.

is equal to $\aleph_0$ and the other is
uncountable, or that
In each case, $(D,E)$ is strongly
asymmetric. This proves that $G$ is symmetrically d-complete.

\pars
For the converse, assume that $G$ is symmetrically d-complete and that
$(\ovl{D},\ovl{E})$ is any cut in $G/\Z$. Then we pick a set $D'\subset
G$ of representatives for $\ovl{D}$ and a set $E'\subset E$ of
representatives for $\ovl{E}$. With $D=D'+\Z$ and $E=E'+\Z$ we obtain a
nonprincipal cut $(D,E)$ in $G$ with image $(\ovl{D},\ovl{E})$ in
$G/\Z$. As before, (\ref{relcof}) holds. By our assumption on $G$, the
cut $(D,E)$ is strongly asymmetric. This implies that at least one of
$\cf(\ovl{D})$ and $\ci(\ovl{E})$ is uncountable and that if both are,
then they are not equal. This shows that $(\ovl{D},\ovl{E})$ is strongly
asymmetric, which proves that $G/\Z$ is strongly symmetrically complete.

\pars
The last equivalence in the theorem is seen as follows. If $G$ is
extremely symmetrically d-complete, then it cannot be isomorphic to $\Z$
and hence, $G/\Z$ is nontrivial. But then, the cofinality of $G$ is
equal to that of $G/\Z$.                                    \qed

%
%
\section{Construction of symmetrically complete linearly ordered
sets}                                       \label{sectcon}
In this section, we will use ``$+$'' in a different way than before. If
$I$ and $J$ are ordered sets, then $I+J$ denotes composition of $I$
and $J$, that is, the disjoint union $I\mathbin{\dot{\cup}} J$ with the
extension of the orderings of $I$ and $J$ given by $i<j$ for all $i \in
I$, $j\in J$.

For any linearly ordered set $I=(I,<)$, we denote by $I^c$ its
completion. Note that $\Coin(I^c)=\Coin(I)$ and $\Cofin(I^c)=\Cofin(I)$.
Further, we denote by $I^*$ the set $I$ endowed with the inverted
ordering $<^*$, where $i<^* j\Leftrightarrow j<i$. If $I'$ is another
ordered set, then $I+I'$ is the {\emph sum} in the sense of order
theory, that is, the orderings of $I$ and $I'$ are extended to $I\cup
I'$ in the unique way such that $i<i'$ for all $i\in I$ and $i'\in I'$.

\pars
We choose some ordered set $I$ (where $I=\emptyset$ is allowed) and
infinite regular cardinals $\mu$ and $\kappa_\nu\,,\, \lambda_\nu$ for
all $\nu<\mu$. We define
\[
I_0\>:=\>\lambda_0^*+I^c+\kappa_0 \quad\mbox{and}\quad
I_\nu:=\lambda_\nu^*+\kappa_\nu\;\mbox{for $0<\nu<\mu$.}
\]
Note that all $I_\nu\,$, $\nu<\mu$, are cut complete. Note further that
if $C$ is a cut in $I_0$ with cofinality $(\kappa,\lambda)$, then
$\kappa\in\Cofin (I)\cup \Reg_{<\kappa_0}$ and $\lambda
\in \Coin (I)\cup \Reg_{<\lambda_0}$.

We define $J$ to be the lexicographic product over the $I_\nu$ with
index set $\mu$; that is, $J$ is the set of all sequences
$(\alpha_\nu)_{\nu<\mu}$ with $\alpha_\nu\in I_\nu$ for all $\nu<\mu$,
endowed with the following ordering: if $(\alpha_\nu)_{\nu<\mu}$ and
$(\beta_\nu)_{\nu<\mu}$ are two different sequences, then there is a
smallest $\nu_0<\mu$ such that $\alpha_{\nu_0} \ne \beta_{\nu_0}$ and we
set $(\alpha_\nu)_{\nu<\mu}< (\beta_\nu)_{\nu<\mu}$ if $\alpha_{\nu_0}<
\beta_{\nu_0}\,$.

\begin{theorem}
The cofinalities of the cuts of $J$ are:
\sn
$(1,\mu)\,,\,(\mu,1)$,
\n
$(\kappa_1,\lambda)\,,\,(\kappa,\lambda_1)$ for $\lambda\in \Coin
(I)\cup\Reg_{<\lambda_0}\,,\,\kappa\in\Cofin (I)\cup \Reg_{<\kappa_0}$,
\n
$(\kappa_{\nu+1},\lambda)\,,\,(\kappa,\lambda_{\nu+1})$ for $0<\nu<\mu$
and $\kappa< \kappa_\nu$, $\lambda<\lambda_\nu$ regular cardinals,
\n
$(\kappa_\nu,\lambda_\nu)$ for $\nu<\mu$ a successor ordinal, and
\n
$(\kappa_\nu,\mu')\,,\,(\mu',\lambda_\nu)$ for $\nu<\mu$ a limit ordinal
and $\mu'<\mu$ its cofinality.
\sn
Further, the cofinality of $J$ is $\kappa_0$ and its coinitiality is
$\lambda_0\,$.
\end{theorem}
\begin{proof}
Take any cut $(D,E)$ in $J$. Assume first that $D$ has a maximal element
$(\alpha_\nu)_{\nu<\mu}\,$. By our choice of the linearly ordered sets
$I_\nu$ we can choose, for every $\nu<\mu$, some $\beta_\nu\in I_\nu$
such that $\beta_\nu>\alpha_\nu\,$. For $\rho<\mu$ we define
$\beta^\rho_\rho:=\beta_\rho$ and $\beta^\rho_\nu :=\alpha_\nu$ for
$\nu\ne\rho$. Then the elements $(\beta^\rho_\nu)_{\nu<\mu}$,
$\rho<\mu$, form a strictly decreasing coinitial sequence of elements in
$E$. Since $\mu$ was chosen to be regular, this shows that the
cofinality of $(D,E)$ is $(1,\mu)$. Similarly, it is shown that if $E$
has a minimal element, then the cofinality of $(D,E)$ is $(\mu,1)$.

\pars
Now assume that $(D,E)$ is nonprincipal. Take $S$ to be the set of all
$\nu'<\mu$ for which there exist $a_{\nu'}=(\alpha_{\nu',\nu})_{\nu<\mu}
\in D$ and $b_{\nu'}= (\beta_{\nu',\nu})_{\nu<\mu}\in E$ such that
$\alpha_{\nu',\nu}= \beta_{\nu',\nu}$ for all $\nu\leq\nu'$. Note that
$S$ is a proper initial segment of the set $\mu$. We claim that
$\nu_1<\nu_2\in S$ implies that
\[
\alpha_{\nu_1,\nu}\>=\>\alpha_{\nu_2,\nu}\;\;
\mbox{ for all }\>\nu\leq\nu_1\>,
\]
or in other words, $(\alpha_{\nu_1,\nu})_{\nu\leq\nu_1}$ is a truncation
of $a_{\nu_2}\,$. Indeed, suppose that this were not the case. Then
there would be some $\nu'<\nu_1$ such that
\[
\beta_{\nu_1,\nu'}=\alpha_{\nu_1,\nu'}\ne
\alpha_{\nu_2,\nu'}=\beta_{\nu_2,\nu'}\>.
\]
Suppose that $\nu'$ is minimal with this property and that the left hand
side is smaller. But then, $(\beta_{\nu_1,\nu})_{\nu<\mu}<(\alpha_{\nu_2,
\nu})_{\nu<\mu}\,$, so $b_{\nu_1}\in D$, a contradiction. A similar
contradiction is obtained if the right hand side is smaller.

Now take $\mu_0$ to be the minimum of $\mu\setminus S$; in fact, $S$ is
is equal to the set $\mu_0\,$. We define
\begin{eqnarray*}
D_{\mu_0} & := & \{\alpha\in I_{\mu_0}\mid\exists\,
(\alpha_\nu)_{\nu<\mu}\in D:\, \alpha_{\mu_0}=\alpha\mbox{ and }
\alpha_\nu=\alpha_{\nu,\nu}\mbox{ for }\nu<\mu_0\}\>,\\
E_{\mu_0} & := & \{\beta\in I_{\mu_0}\mid\exists
(\beta_\nu)_{\nu<\mu}\in E:\, \beta_{\mu_0}=\beta\mbox{ and }
\beta_\nu=\alpha_{\nu,\nu}\mbox{ for }\nu<\mu_0\}\>.
\end{eqnarray*}
By our definition of $\mu_0\,$, these two sets are disjoint, and it is
clear that their union is $I_{\mu_0}$ and every element in $D_{\mu_0}$
is smaller than every element in $E_{\mu_0}$. However, one of the sets
may be empty, and we will first consider this case. Suppose that
$E_{\mu_0}=\emptyset$. Then $D_{\mu_0}=I_{\mu_0}$ and since this has no
last element, the cofinality of $D$ is the same as that of
$I_{\mu_0}\,$, which is $\kappa_{\mu_0}$. In order to determine the
coinitiality of $E$, we proceed as in the beginning of this proof.
Observe that since $E_{\mu_0}=\emptyset$, for an element
$(\beta^\rho_\nu)_{\nu<\mu}$ to lie in $E$ it is necessary that
$\beta^\rho_\nu>\alpha_{\nu,\nu}$ for some $\nu<\mu_0\,$.
For all $\nu<\mu_0$, we choose some $\beta_\nu\in I_\nu$ such that
$\beta_\nu> \alpha_{\nu,\nu}\,$; then for all $\rho<\mu_0$ we define
$\beta^\rho_\rho:=\beta_\rho$, $\beta^\rho_\nu :=\alpha_{\nu,\nu}$ for
$\nu<\rho$, and choose $\beta^\rho_\nu$ arbitrarily for
$\rho<\nu<\mu$. Then the elements $(\beta^\rho_\nu)_{\nu<\mu}$,
$\rho<\mu_0\,$, form a strictly decreasing coinitial sequence in
$E$. If $\mu'$ denotes the cofinality of $\mu_0\,$, this shows that the
coinitiality of $E$ is $\mu'$, and the cofinality of $(D,E)$ is
$(\kappa_{\mu_0},\mu')$. Since $\mu$ was chosen to be regular, we have
that $\mu'<\mu$.

Similarly, it is shown that if $D_{\mu_0}=\emptyset$, then the
cofinality of $(D,E)$ is $(\mu',\lambda_{\mu_0})$. Note that $D_{\mu_0}$
or $E_{\mu_0}$ can only be empty if $\mu_0$ is a limit ordinal. Indeed,
if $\mu_0=0$ and $(\alpha_\nu)_{\nu<\mu}\in D$, $(\beta_\nu)_{\nu<\mu}
\in E$, then $\alpha_0\in D_0$ and $\beta_0\in E_0\,$; if $\mu_0=
\mu'+1$, then with $(\alpha_{\mu',\nu})_{\nu<\mu} \in D$ and
$(\beta_{\mu',\nu})_{\nu<\mu}\in E$ chosen as before, it follows that
$\alpha_{\mu',\mu_0}\in D_{\mu_0}$ and $\beta_{\mu',\mu_0}\in
E_{\mu_0}\,$.

\pars
From now on we assume that both $D_{\mu_0}$ and $E_{\mu_0}$ are
nonempty.
Since $I_{\mu_0}$ is complete, $D_{\mu_0}$ has a maximal element
or $E_{\mu_0}$ has a minimal element.



Suppose that $D_{\mu_0}$ has a maximal element $\tilde{\alpha}$. Then
for all $\rho\in\kappa_{\mu_0+1}\subset I_{\mu_0+1}\,$, we define
$\alpha^\rho_\nu= \alpha_{\nu,\nu}$ for $\nu<\mu_0\,$,
$\alpha^\rho_{\mu_0} =\tilde{\alpha}$, $\alpha^\rho_{\mu_0+1}=\rho$, and
choose an arbitrary element of $I_\nu$ for $\alpha^\rho_\nu$ when
$\mu_0+1<\nu <\mu\,$. Then the elements $(\alpha^\rho_\nu)_{\nu<\mu}$,
$\rho\in \kappa_{\mu_0+1}\,$, form a strictly increasing cofinal
sequence in $D$. Since $\kappa_{\mu_0+1}$ was chosen to be a regular
cardinal, this shows that the cofinality of $D$ is $\kappa_{\mu_0+1}$.

Suppose that $E_{\mu_0}$ has a minimal element $\tilde{\beta}$. Then for
every $\sigma\in \lambda_{\mu_0+1}^*\subset I_{\mu_0+1}\,$, we define
$\beta^\sigma_\nu= \alpha_{\nu,\nu}$ for $\nu\leq\mu_0\,$,
$\beta^\sigma_{\mu_0}= \tilde{\beta}$, $\beta^\sigma_{\mu_0+1} =\sigma$,
and choose an arbitrary element of $I_\nu$ for $\beta^\sigma_\nu$ when
$\mu_0+1< \nu<\mu\,$. Then the elements $(\beta^\sigma_\nu)_{\nu<\mu}$,
$\sigma \in\lambda_{\mu_0+1}^*\,$, form a strictly decreasing coinitial
sequence in $E$. Since $\lambda_{\mu_0+1}$ was chosen to be a regular
cardinal, this shows that the coinitiality of $E$ is $\lambda_{\mu_0+1}$.

If $D_{\mu_0}$ has a maximal element and $E_{\mu_0}$ has a minimal
element, then we obtain that the cofinality of $(D,E)$ is
$(\kappa_{\mu_0+1},\lambda_{\mu_0+1})$.

\pars
Now we deal with the case where $D_{\mu_0}$ does not have a maximal
element. Since $I_{\mu_0}$ is complete, $E_{\mu_0}$ must then have a
smallest element, and by what we have already shown, we find that $E$
has coinitiality $\lambda_{\mu_0+1}$. Denote the cofinality of
$D_{\mu_0}$ by $\kappa$. We choose a sequence of elements
$\alpha^\rho_{\mu_0}$, $\rho<\kappa$, cofinal in $D_{\mu_0}\,$. For all
$\rho<\kappa$, we define $\alpha^\rho_\nu= \alpha_{\nu,\nu}$ for
$\nu<\mu_0\,$ and choose an arbitrary element of $I_\nu$ for
$\alpha^\rho_\nu$ when $\mu_0+1<\nu <\mu\,$. Then the elements
$(\alpha^\rho_\nu)_{\nu<\mu}$, $\rho<\kappa$, form a strictly increasing
cofinal sequence in $D$. Hence, $(D,E)$ has cofinality $(\kappa,
\lambda_{\mu_0+1})$ with $\kappa$ the cofinality of a lower cut set in
$I_{\mu_0}$, i.e., $\kappa\in\Cofin (I)\cup \Reg_{<\kappa_0}$ if
$\mu_0=0$, and $\kappa\in\Reg_{<\kappa_{\mu_0}}$ otherwise.

If $E_{\mu_0}$ does not have a minimal element, then a symmetrical
argument shows that the cofinality of $(D,E)$ is $(\kappa_{\mu_0+1},
\lambda)$ for some $\lambda$ the coinitiality of an upper cut set in
$I_{\mu_0}$, i.e., $\lambda\in \Coin (I)\cup \Reg_{<\lambda_0}$ if
$\mu_0=0$, and $\lambda\in\Reg_{<\lambda_{\mu_0}}$ otherwise.



\pars
We have now proved that the cofinalities of the cuts in $J$ are all
among those listed in the statement of the theorem. By our arguments it
is also clear that all listed cofinalities do indeed appear.

\pars
Finally, the easy proof of the last statement of the theorem is left to
the reader.
\end{proof}

The following result is an immediate consequence of the theorem:
\begin{corollary}
Assume that\sn
a) \ $\kappa_1\notin \Coin (I)\cup\Reg_{<\lambda_0}$ and $\lambda_1\notin
\Cofin (I)\cup \Reg_{<\kappa_0}$,\n
b) \ $\kappa_{\nu+1}\geq\lambda_\nu$ and $\lambda_{\nu+1}\geq\kappa_\nu$
for all $\nu<\mu$,\n
c) \ $\kappa_\nu\ne\lambda_\nu$ for $\nu<\mu$ a successor ordinal, and\n
d) \ $\kappa_\nu\geq\mu$ and $\lambda_\nu\geq\mu$ for $\nu<\mu$ a limit
ordinal.
\sn
Then $J$ is symmetrically complete. If in addition $\mu$ is uncountable,
then $J$ is strongly symmetrically complete, and if also $\kappa_0$ and
$\lambda_0$ are uncountable, then $J$ is extremely symmetrically
complete.
\end{corollary}

It is easy to choose our cardinals by transfinite induction in such a
way that all conditions of this corollary are satisfied. We
choose
\sn
$\bullet$ \ $\kappa_0$ and $\lambda_0$ to be arbitrary uncountable
regular cardinals,
\n
$\bullet$ \ $\mu>\max\{\kappa_0,\lambda_0,\card(I)\}$,
\n
$\bullet$ \ $\kappa_\nu=\mu$ and $\lambda_\nu=\mu^+$ for $\nu=1$ or
$\nu<\mu$ a limit ordinal,
\n
$\bullet$ \ $\kappa_{\nu+1}=\kappa_\nu^{++}$ and $\lambda_{\nu+1}=
\lambda_\nu^{++}$ for $0<\nu<\mu$.

\pars
Sending an element $\alpha\in I$ to an arbitrary element
$(\alpha_\nu)_{\nu<\mu}\in J$ with $\alpha_0=\alpha$ induces an order
preserving embedding of $I$ in $J$. So we obtain the following result:

\begin{corollary}                               \label{embeds}
Every linearly ordered set $I$ can be embedded in an extremely
symmetrically complete ordered set $J$.
\end{corollary}

\parb
Our above construction can be seen as a ``brute force'' approach. We
will now present a construction that offers more choice for the
prescribed cofinalities.

If an index set $I$ is not well ordered, then the lexicographic product
of ordered abelian groups $G_i\,$, $i\in I$, is defined to be the subset
of the product consisting of all elements $(g_i)_{i\in I}$ with well
ordered support $\{i\in I\mid g_i\ne 0\}$. Likewise, the lexicographic
sum is defined to be the subset consisting of all elements $(g_i)_{i\in
I}$ with finite support $\{i\in I\mid g_i\ne 0\}$. The problem with
ordered sets is that they ususally do not have distinguished elements
(like neutral elements for an operation). The remedy used in
\cite{[KKS]} is to fix distinguished elements in all linear orderings we
wish to use for our lexicographic sum. Hausdorff (\cite{[Hd]}) does this
in quite an elegant way: he observes that the full product is still
partially ordered. Singling out one element in the product then
determines the distinguished elements in the ordered sets (being the
corresponding components of the element), and in this manner one obtains
an associated maximal linearly ordered subset of the full product.

While the index sets we use here are ordinals and hence well ordered,
which makes a condition on the support unnecessary for the work with
lexicographic products, we will use the idea (as apparent in the
definition of the lexicographic sum) that certain elements can be
singled out by means of their support.

\pars
We choose infinite regular cardinals $\mu$, $\kappa_0$ and
$\lambda_0\,$. Further, we denote by $\On$ the class of all ordinals and
%
set
\[
I_0\>:=\>\lambda_0^*+I^c+\kappa_0 \quad\mbox{and}\quad
I_\nu:=\On^*+\mu+\{\zero\}+\mu^*+\On\;\;
\mbox{ for $0<\nu<\mu$,}
\]
assuming that $\zero$ does not appear in $I^c$ or any ordinal or
reversed ordinal. Note that On can be replaced by a large enough
cardinal; its minimal size depends on the choice of $I$, $\mu$,
$\kappa_0$ and $\lambda_0\,$. But the details are not essential
for our construction, so we skip them.

We define $J^\circ$ to consist of all elements of the lexicographic
product over the $I_\nu$ with index set $\mu$ whose support
\[
\supp (\alpha_\nu)_{\nu<\mu}\>=\>\{\nu\mid \nu<\mu\mbox{ and }
\alpha_\nu\ne \zero\}
\]
is an initial segment of $\mu$ (i.e., an ordinal $\leq\mu$).

A further refinement of our construction uses the idea to define
suitable subsets of $J^\circ$ by restricting the choice of the
coefficient $\alpha_\nu$ in dependence on the truncated sequence
$(\alpha_{\rho})_{\rho<\nu}\,$.

For every $\nu<\mu$ we consider the following set of truncations:
\[
J^\circ_\nu\>:=\>\{(\alpha_\nu)_{\rho\leq\nu}\mid
(\alpha_\nu)_{\rho<\mu}\in J^\circ\}\>.
\]
By induction on $\nu<\mu$ we define subsets
\[
J_{\nu}\>\subset\>J^\circ_\nu
\]
as follows:
\sn
(J1) \ $J_0:=I_0\,$.
\sn
(J2) \ If $J_{\nu'}$ for all $\nu'<\nu$ are already constructed, then we
first define the auxiliary set
\[
J_{<\nu}\>:=\>\{(\alpha_\rho)_{\rho<\nu}\in J^\circ_\nu\mid
(\alpha_\rho)_{\rho<\nu'}\in J_{\nu'}\mbox{ for all }\nu'<\nu\}\>.
\]
For $a=(\alpha_\rho)_{\rho<\nu}\in J_{<\nu}$ we set
\[
\kappa_a\>:=\>\cf (\{b\in J_{<\nu}\mid b<a\})\;\mbox{ and }
\lambda_a\>:=\>\ci (\{b\in J_{<\nu}\mid b>a\})\>,
\]
and
\begin{equation}                            \label{defInu}
I_\nu(a)\>:=\>
\left\{
\begin{array}{l}
\{\zero\}\;\mbox{ if $\alpha_\rho=0$ for some }\rho<\nu\,,\\
\phir(\kappa_a)^* +\mu+ \{\zero\}+\mu^*+ \phil(\lambda_a) \;
\mbox{ otherwise.}
\end{array}
\right.
\end{equation}
So $I_\nu((\alpha_\rho)_{\rho<\nu})\subset\> I_\nu\,$. Now we let
$(\alpha_\rho)_{\rho\leq\nu}\in J^\circ_\nu$ be an element of $J_\nu$ if
and only if $a=(\alpha_\rho)_{\rho<\nu}\in J_{<\nu}$ and $\alpha_\nu \in
I_\nu(a)$. Note that by our definition, also the following condition
will be satisfied:
\sn
(J3) \ if $\alpha_{\nu}=\zero$ for $\nu<\nu'<\mu$, then
$\alpha_{\nu'}=\zero\,$.

\sn
After having defined $J_{\nu}$ for all $\nu<\mu$, we set
\[
J\>:=\>\{(\alpha_\rho)_{\rho<\mu}\in J^\circ\mid
(\alpha_\rho)_{\rho\leq\nu} \in J_\nu \mbox{ for all }\nu<\mu\}\>.
\]

\pars
The following is our first step towards the proof of
Theorem~\ref{IextJ}:

\begin{theorem}                             \label{propJ}
With the sets $\Rl$ and $\Rr$ defined as in the introduction,
assume that (\ref{phi}) holds.
Then the cofinalities of the cuts of $J$ are:
\[
\{(1,\mu),(\mu,1)\} \,\cup\, \{(\kappa,\phir(\kappa))\mid \kappa\in
\Rl\} \,\cup\, \{(\phil(\lambda),\lambda)\mid\lambda\in \Rr\}\>.
\]
Further, the cofinality of $J$ is $\kappa_0$ and its coinitiality is
$\lambda_0\,$.
\end{theorem}
\begin{proof}
First, we observe that for each $\nu<\mu$ we obtain an embedding
\[
\iota_\nu:\>J_\nu\hookrightarrow J
\]
by sending $(\alpha_\rho)_{\rho\leq\nu}$ to
$(\beta_\rho)_{\rho\leq\mu}$, where $\beta_\rho=\alpha_\rho$ for
$\rho\leq\nu$ and $\beta_\rho=\zero$ for $\nu<\rho<\mu$.

We start by proving that the principal cuts in $J$ have cofinalities
$(1,\mu)$ and $(\mu,1)$. Take $(\alpha_\rho)_{\rho\leq\mu}\in J$ and
assume first that its support is smaller than $\mu$. Set $\nu:=\min
\{\rho<\mu\mid\alpha_\rho=\zero\}\geq 1$. Then the second case of
definition (\ref{defInu}) applies and therefore, the
cofinalities of the principal cuts generated by $\zero$ in $I_\nu$ are
$(1,\mu)$ and $(\mu,1)$ and therefore, the cofinalities of the principal
cuts generated by $(\alpha_\rho)_{\rho\leq\nu}$ in $J_\nu$ are also
$(1,\mu)$ and $(\mu,1)$. By means of the embeddings $\iota_\nu$ it
follows that the cofinalities of the principal cuts generated by
$(\alpha_\rho)_{\rho<\mu}$ in $J$ are again $(1,\mu)$ and $(\mu,1)$.

Now assume that the support of $a:=(\alpha_\rho)_{\rho<\mu}$ is $\mu$.
For each $\nu<\mu$ there are elements $\beta_\nu,\gamma_\nu\in I_\nu
((\alpha_\rho)_{\rho<\nu})$ with $\beta_\nu<\alpha_\nu<\gamma_\nu$. We
set $\beta_\rho:=\gamma_\rho:= \alpha_\rho$ for $\rho<\nu$, and define
\[
b_\nu\>:=\>\iota_\nu((\beta_\rho)_{\rho\leq\nu}) \;\mbox{ and }\;
c_\nu\>:=\>\iota_\nu((\gamma_\rho)_{\rho\leq\nu})\>.
\]
Then we find that whenever $\nu<\nu'<\mu$, then
\[
b_\nu\><\>b_{\nu'}\><\>a\><\>a_{\nu'}\><\>a_{\nu}\>.
\]
This proves that again, the cofinalities of the principal cuts generated
by $(\alpha_\rho)_{\rho<\mu}$ in $J$ are $(1,\mu)$ and $(\mu,1)$.

\pars
Now take any nonprincipal cut $(D,E)$ in $J$. By restricting the
elements to index set $\nu+1=\{\rho\mid \rho\leq\nu\}$, this cut
induces a quasicut $(D_\nu,E_\nu)$ in $J_\nu$, that is, $J_\nu=D_\nu\cup
E_\nu\,$, $D_\nu<E_\nu\,$, and therefore, $D_\nu\cap E_\nu$ contains at
most one element.

Assume that $\nu<\mu$ is such that $\iota_\nu(D_\nu)$ is not a cofinal
subset of $D$ and $\iota_\nu(E_\nu)$ is not a coinitial subset of $E$.
Then we have one of the following cases:
\sn
$\bullet$ \ $\iota_\nu(D_\nu)\cap E\ne\emptyset$ or
$\iota_\nu(E_\nu)\cap D\ne\emptyset$,
\sn
$\bullet$ \ there are $d_\nu\in D$ and $e_\nu\in E$ such that
$\iota_\nu(D_\nu)<d_\nu<e_\nu< \iota_\nu(E_\nu)$, which yields that the
restrictions of $d_\nu$ and $e_\nu$ to index set $\nu+1$ are equal and
lie in $D_\nu\cap E_\nu$.
\sn
In both cases, $D_\nu\cap E_\nu\ne\emptyset$. This implies that also
$D_{\nu'}\cap E_{\nu'}\ne \emptyset$ for all $\nu' < \nu$, with the
element in $D_{\nu'}\cap E_{\nu'}$ being the restriction of the element
in $D_\nu\cap E_\nu$.

Now we show that there is some $\nu<\mu$ such that $\iota_\nu(D_\nu)$
is cofinal in $D$ or $\iota_\nu(E_\nu)$ is coinitial in $E$. Suppose
that the contrary is true. Then
$D_\nu\cap E_\nu\ne\emptyset$ for all $\nu<\mu$ and there is a unique
element $a\in J$ whose restriction to index set $\nu+1$ lies in $D_\nu
\cap E_\nu\,$, for all $\nu<\mu$. It follows that $a$ is either the
largest element of $D$ or the smallest element in $E$. But this
contradicts our assumption that $(D,E)$ is nonprincipal.

We take $\nu$ to be minimal with the property that $\iota_\nu
(D_\nu)$ is cofinal in $D$ or $\iota_\nu(E_\nu)$ is coinitial in $E$.
From what we have shown above, it follows that $D_{\nu'}\cap E_{\nu'}\ne
\emptyset$ for all $\nu'<\nu$ and there is $(\alpha_\rho)_{\rho<\nu}\in
J_{<\nu}$ whose restriction to $\nu'+1$ lies in $D_{\nu'}\cap E_{\nu'}$,
for all $\nu'<\nu$. Therefore, there must be elements in both $D_\nu$
and $E_\nu$ whose restrictions to $\nu$ are equal to
$(\alpha_\rho)_{\rho<\nu}$. Consequently, with
\pars
$\ovl{D}_\nu\>:=\>\{\alpha_\nu\in I_\nu((\alpha_\rho)_{\rho<\nu})\mid
(\alpha_\rho)_{\rho\leq\nu}\in D_\nu\}$ and
\par
$\ovl{E}_\nu\>:=\>\{\alpha_\nu\in I_\nu((\alpha_\rho)_{\rho<\nu})\mid
(\alpha_\rho)_{\rho\leq\nu}\in E_\nu\}$,
\sn
$(\ovl{D}_\nu,\ovl{E}_\nu)$ is a cut in $I_\nu((\alpha_\rho)_{\rho<\nu})
\,$. But $I_\nu ((\alpha_\rho)_{\rho<\nu})$ is cut complete, and so
there is some $\alpha_\nu\in I_\nu((\alpha_\rho)_{\rho<\nu})$ such that
$a=(\alpha_\rho)_{\rho\leq\nu}$ is either the largest element of $D_\nu$
or the smallest element of $E_\nu\,$. We note that $\alpha_\nu\ne
\zero\,$; otherwise, the element $(\alpha_\rho)_{\rho<\mu}$ with
$\alpha_\rho= \zero$ for $\nu\leq\rho<\mu$, which is the unique element
in $J$ whose restriction to $\nu+1$ is $a$, would be the largest element
of $D$ or the smallest element of $E$ in contradiction to our assumption
on $(D,E)$. Hence by construction, for every
\[
\alpha\in I_{\nu+1}((\alpha_\rho)_{\rho\leq\nu})\>=\>
\phir(\kappa_a)^* +\mu+ \{\zero\}+\mu^*+ \phil(\lambda_a)
\]
there is an element $(\alpha_\rho)_{\rho<\mu}$ with $\alpha_{\nu+1}=
\alpha$ whose restriction to $\nu+1$ is $a$.

\pars
We assume first that $\iota_\nu(D_\nu)$ is cofinal in $D$. Since $(D,E)$
is nonprincipal, $D$ and hence also $D_\nu$ has no largest element. So
$a$ is the smallest element of $E_\nu\,$. Consequently,
\[
\iota_{\nu+1}(\{(\alpha_\rho)_{\rho\leq\nu+1}\mid\alpha_{\nu+1}\in
I_{\nu+1}((\alpha_\rho)_{\rho\leq\nu})\})
\]
is coinitial in $E$. We observe that $\kappa_a=
\cf(D_\nu)=\cf(D)\ne 1$. By construction, the coinitiality of $I_{\nu+1}
((\alpha_\rho)_{\rho\leq\nu})$ is $\phir(\kappa_a)$. This proves that
the cofinality of $(D,E)$ is $(\kappa_a,\phir(\kappa_a))$.

If on the other hand, $\iota_\nu(E_\nu)$ is coinitinal in $E$, then
$a$ is the largest element of $D_\nu$ and one shows along the same lines
as above that the cofinality of $(D,E)$ is $(\phil(\lambda_a),\lambda_a)$
with $\lambda_a=\ci(E_\nu)=\ci(E)\ne 1$.

We have to prove that the cardinals $\kappa_a$ and $\lambda_a$ that
appear in the construction, i.e., in definition (\ref{defInu}), are
elements of $\{1\}\cup\Rl$ and $\{1\}\cup\Rr$, respectively. We observe
that $\kappa_a$ and $\lambda_a$ appear in definition (\ref{defInu}) only
if $a=(\alpha_\rho)_{\rho<\nu}\in J_{<\nu}$ is such that $\alpha_\rho\ne
\zero$ for all $\rho<\nu$. We show our assertion by induction on
$1\leq\nu\leq\mu$. We do this for $\kappa_a\,$; for $\lambda_a$ the
proof is similar. First, we consider the successor case $\nu=\sigma+1$.
We set $\ovl{a}= (\alpha_\rho)_{\rho<\sigma})$. If $\sigma\geq 1$, then
our induction hypothesis states that our assertion is true for
$\kappa_{\ovl{a}}$ and $\lambda_{\ovl{a}}\,$. We observe that because
the second case of definition (\ref{defInu}) applies to $\ovl{a}$,
\begin{eqnarray*}
\kappa_a & = & \cf(\{(\beta_\rho)_{\rho\leq\sigma}\in J_\sigma\mid
\beta_\rho= \alpha_\rho\mbox{ for }\rho<\sigma\mbox{ and }
\beta_{\sigma}< \alpha_\sigma\}) \\
 & = & \cf(\{\beta\in I_\sigma(\ovl{a})\mid\beta< \alpha_\sigma\})\>.
\end{eqnarray*}
This is the cofinality of a lower cut set of a cut in $I_\sigma$.
Therefore, if $\kappa_a$ is infinite, it is an element of $\Cofin (I)
\cup\Reg_{<\kappa_0}\cup \Reg_{<\mu}=\Rl$ if $\sigma=0$, and of
$\Reg_{<\phil(\lambda_{\ovl{a}})} \cup \Reg_{<\mu}$ otherwise. In the
latter case, $\lambda_{\ovl{a}}\in\Rr$ by induction hypothesis, hence
$\phil(\lambda_{\ovl{a}})\in \Rl$ by (\ref{phi}), which yields that
$\Reg_{<\phil(\lambda_{\ovl{a}})} \cup \Reg_{<\mu}\subseteq\Rl$.
Altogether, we have proved that $\kappa_a\in\{1\}\cup\Rl$.

\pars
Now we consider the case of $\nu$ a limit ordinal. Let $\mu'$ be its
cofinality. Then $\mu'\in \Reg_{<\mu}\,$. With a similar construction as
in the beginning of the proof one shows that the principal cuts
generated by elements in $J_{<\nu}$ have cofinalities $(\mu',1)$ and
$(1,\mu')$. This yields that $\kappa_a\in\Rl$ and $\lambda_a\in\Rr$.


\parm
It remains to prove that all cofinalities listed in the assertion of our
theorem actually appear as cofinalities of cuts in $J$. Since for all
cardinals $\kappa\in \Cofin (I) \cup \Reg_{<\kappa_0}$, there is a cut
in $I_0$ with cofinality $(\kappa,1)$, our construction at level $\nu=1$
shows that $(\kappa, \phir(\kappa))$ appears as the cofinality of a cut
in $J$. Similarly, one shows that $(\phil(\lambda),\lambda)$ appears as
the cofinality of a cut in $J$ for every $\lambda\in\Coin (I)\cup
\Reg_{<\lambda_0}$.

Now take any regular cardinal $\mu'<\mu$. For an arbitrary
$a=(\alpha_0)\in J_0$ we see that the second case of definition
(\ref{defInu}) applies to $I_1(a)$, so that there is a cut in $I_1$ with
cofinality $(\mu',1)$. Our construction at level $\nu=2$ then shows that
$(\mu', \phir(\mu'))$ appears as the cofinality of a cut in $J$.
Similarly, one shows that $(\phil(\mu'),\mu')$ appears as the cofinality
of a cut in $J$.


\pars
The proof of the last statement of the theorem is again left to
the reader.
\end{proof}

The following result is an immediate consequence of
Theorem~\ref{propJ}, and it proves Theorem~\ref{IextJ}:
\begin{corollary}
Assume in addition to the previous assumptions
that $\phil(\kappa)\ne\kappa\ne\phir(\kappa)$ for all $\kappa\in\Rl\cup
\Rr$. Then $J$ is a symmetrically complete extension of $I$. If in
addition $\mu$ is uncountable, then $J$ is strongly symmetrically
complete.
\end{corollary}

\begin{remark}
In both constructions that we have given in this section, every element
in the constructed ordered set has, in the terminology of Hausdorff,
character $(\mu,\mu)$.
\end{remark}

%
%
\section{Construction of symmetrically complete ordered
extensions}                                            \label{sectemb}
Take any ordered abelian group $G$. We wish to extend it to an extremely
symmetrically complete ordered abelian group. We use the well known fact
that $G$ can be embedded in a suitable Hahn product $H_0=\pH_I \R$, for
some ordered index set $I$. By Corollary~\ref{embeds}, there is an
embedding $\iota$ of $I$ in an extremely symmetrically complete linearly
ordered set $J$. We set $H=\pH_J \R$ and note that there is a canonical
order preserving embedding $\varphi$ of $H_0=\pH_I \R$ in $H=\pH_J \R$
which lifts $\iota$ by sending an element $(r_\gamma)_{\gamma\in I}$ to
the element $(r'_\delta)_{\delta\in J}$ where $r'_\delta=r_\gamma$ if
$\delta= \iota(\gamma)$ and $r'_\delta=0$ if $\delta$ is not in the
image of $\iota$. By Theorem~\ref{MTag}, $H$ is an extremely
symmetrically complete ordered abelian group.
We have now proved the first part of Theorem~\ref{ext}.

\parm
Take any ordered field $K$. We wish to extend it to an extremely
symmetrically complete ordered field. First, we extend $K$ to its real
closure $K^{\mbox{\tiny\rm rc}}$. From [Ka] we know that
$K^{\mbox{\tiny\rm rc}}$ can be embedded in the power series field
$\R((G))$ where $G$ is the value group of $K^{\mbox{\tiny\rm rc}}$ under
the natural valuation. By what we have already shown, $G$ admits an
embedding $\psi$ in an extremely symmetrically complete ordered abelian
group $H$. By a definition analogous to the one of $\varphi$ above, one
lifts $\psi$ to an order preserving embedding of the power series field
$\R((G))$ in the power series field $\R((H))$. By Theorem~\ref{MTf},
$\R((H))$ is an extremely symmetrically complete ordered field.
We have thus proved the second part of Theorem~\ref{ext}.

\newcommand{\lit}[1]{\bibitem{#1}}

\end{document}